\documentclass[11pt]{article}

\usepackage{amsmath,amssymb}
\usepackage{amsthm} 
\usepackage{newtxtext}
\usepackage{newtxmath}

\usepackage[T1]{fontenc}
\usepackage[utf8]{inputenc}
\usepackage[margin=1in]{geometry}

\newtheorem{theorem}{Theorem}
\newtheorem{lemma}{Lemma}
\newtheorem{proposition}{Proposition}

\theoremstyle{remark}
\newtheorem{remark}{Remark}

\title{A note on bistability of a two-gene competitive system}
\author{Eduardo D. Sontag%
\thanks{This work was partially supported by grants AFOSR
FA9550-21-1-0289 and FA9550-22-1-0316.}%
\thanks{Department of Electrical and Computer Engineering and Department
of Bioengineering, Northeastern University, Boston, Massachusetts 02115,
USA.  Email: \texttt{e.sontag@northeastern.edu}}}
\date{}

\begin{document}

\maketitle

\begin{abstract}
Self-regulation together with mutual promoter competition provides a
simple mechanism for bistability in gene-regulatory models.  We study
a two-gene system with regulatory terms of Hill exponent one, allowing
distinct basal production rates and distinct degradation rates.  Each
gene product, when bound to its own promoter, may enhance or reduce
production relative to the basal rate, while the two products compete
through promoter occupancy.
We show that the system has at least one and at most three equilibria
in the positive quadrant.  Exactly two positive equilibria can occur
only if one nullcline intersection is degenerate; consequently, a
configuration in which all positive nullcline intersections are
transverse has either one or three positive equilibria.  If there are
exactly three distinct positive equilibria, then all three are
automatically hyperbolic: the two outer equilibria are asymptotically
stable nodes and the middle equilibrium is a saddle.  Moreover, every
positive solution converges to an equilibrium.  Hence the positive
quadrant is the disjoint union of the basins of attraction of the two
stable nodes and the one-dimensional stable manifold of the saddle,
yielding global bistability.
\end{abstract}

\section{Introduction}

Two-gene circuits with self-regulation and mutual competition provide a
particularly simple setting in which to study bistability.  In the
model considered here, the production rate associated with occupation
of a promoter by its own gene product may be either larger or smaller
than the basal production rate.  Thus the model includes, as special
cases, direct self-activation, no direct effect, and direct
self-repression upon self-binding.  Even though all binding terms have
Hill exponent one,
the system can have three positive steady states.  This occurs, for
example, in regimes with strong self-activation, and can also occur
without self-activation when cross-binding is sufficiently tighter
than self-binding; explicit examples are given in
Remark~\ref{rem:examples}.

The purpose of this note is to show that the mere existence of three
distinct positive equilibria already forces the expected bistable phase
portrait.  In particular, it is not necessary to solve explicitly for
the equilibria, nor is it necessary to impose a separate
nondegeneracy or transversality hypothesis.  There is also a useful
algebraic fact behind this conclusion: the system always has at least
one and at most three distinct positive equilibria.  If there are
exactly two, one of the two nullcline intersections is necessarily
degenerate.  Thus, if all positive nullcline intersections are
transverse, the only possibilities are one or three positive
equilibria.  This is the precise transversality sense in which the
two-equilibrium case is exceptional.

Consider the system
\begin{align}
\dot x_1
&=
\frac{\alpha_{01}+\alpha_1x_1/K_1}
{1+x_1/K_1+x_2/K_{21}}-\delta_1 x_1,
\label{eq:system1}\\
\dot x_2
&=
\frac{\alpha_{02}+\alpha_2x_2/K_2}
{1+x_2/K_2+x_1/K_{12}}-\delta_2 x_2,
\label{eq:system2}
\end{align}
where
\[
\alpha_{01},\alpha_{02},\delta_1,\delta_2,
K_1,K_2,K_{12},K_{21}>0,
\qquad
\alpha_1,\alpha_2\geq0.
\]
Here $\alpha_{0i}$ is the basal production rate for gene $i$, while
$\alpha_i$ is the production rate associated with occupation of its
promoter by its own gene product.  Thus $\alpha_i>\alpha_{0i}$
corresponds to direct self-activation upon self-binding,
$\alpha_i=\alpha_{0i}$ to no change in production upon self-binding, and
$\alpha_i<\alpha_{0i}$ to direct self-repression upon self-binding.
Since self-binding also displaces the product of the other gene from the
promoter, the net production rate of gene~$i$ need not be decreasing in
$x_i$, even when $\alpha_i<\alpha_{0i}$.  We use the convention that $K_{ij}$ is the
dissociation constant for binding of the product of gene $i$ to the
promoter of gene $j$; hence $K_{21}$ appears in the equation for
gene~1 and $K_{12}$ in the equation for gene~2.

We are interested in equilibria in the open positive quadrant
\[
\mathbb R_{>0}^2=\{(x_1,x_2):x_1>0,\ x_2>0\}.
\]

The main result is the following.

\begin{theorem}\label{thm:bistability}
Suppose that \eqref{eq:system1}--\eqref{eq:system2} has exactly three
distinct equilibria in $\mathbb R_{>0}^2$.  Then they can be labeled
\[
p^{(i)}=(x_1^{(i)},x_2^{(i)}),\qquad i=1,2,3,
\]
so that
\begin{equation}\label{eq:ordering}
x_1^{(1)}<x_1^{(2)}<x_1^{(3)},
\qquad
x_2^{(1)}>x_2^{(2)}>x_2^{(3)}.
\end{equation}
All three equilibria are nondegenerate.  The equilibria $p^{(1)}$ and
$p^{(3)}$ are hyperbolic asymptotically stable nodes, whereas $p^{(2)}$
is a hyperbolic saddle.

Moreover, every solution with initial condition in
$\mathbb R_{>0}^2$ is defined for all $t\geq0$ and converges to one of
these three equilibria.  More precisely, if
\begin{align*}
\mathcal B_i
&:=\left\{x^0\in\mathbb R_{>0}^2:
\lim_{t\to\infty}\phi_t(x^0)=p^{(i)}\right\},
\qquad i=1,3,\\
W^s(p^{(2)})
&:=\left\{x^0\in\mathbb R_{>0}^2:
\lim_{t\to\infty}\phi_t(x^0)=p^{(2)}\right\},
\end{align*}
where $\phi_t$ denotes the flow, then
\begin{equation}\label{eq:global-partition}
\mathbb R_{>0}^2
=\mathcal B_1\,\dot\cup\,W^s(p^{(2)})\,\dot\cup\,\mathcal B_3.
\end{equation}
The sets $\mathcal B_1$ and $\mathcal B_3$ are nonempty open basins of
attraction, and $W^s(p^{(2)})$ is the one-dimensional stable manifold
of the saddle.  Thus every positive solution except those on the stable
manifold of the saddle converges to one of the two asymptotically stable
equilibria.  In this sense the system is globally bistable.
\end{theorem}

The proof has two ingredients.  First, the two nullclines have a very
restricted geometry: each is a strictly decreasing graph, and their
intersection condition reduces to a polynomial of degree at most four.
The endpoint signs of the two graphs then give a parity constraint on
the possible intersections.  Second, the system is strongly
competitive and every positive orbit has compact closure in the open
positive quadrant, so the planar convergence theorem for competitive
systems applies.

In Section~\ref{sec:general} we examine the role of the degree bound.
Under purely qualitative hypotheses (strong competition, decreasing
nullclines, and the endpoint signs), three positive equilibria lead to
one of two alternatives (Theorem~\ref{thm:general}).  Either the two
nullclines cross at all three equilibria, in which case the
configuration of Theorem~\ref{thm:bistability} follows whenever the
equilibria are nondegenerate, or they cross at only one of them and
touch without crossing at the other two, which are then degenerate.  An
explicit polynomial example shows that the second alternative does
occur and that bistability then fails; for
\eqref{eq:system1}--\eqref{eq:system2}, the degree-four bound is what
excludes it (Proposition~\ref{prop:special-case}).

Models with this type of competitive promoter structure have appeared
in the literature on genetic switches.  In particular, Kepler and
Elston~\cite{KeplerElston2001} considered a two-gene mutual-repressor
system in which the two transcription factors compete for a shared
operator.  In the deterministic fast-binding limit, this leads to
production terms having the same competitive-denominator structure as
those considered here.  Their model assumes dimeric binding, and hence
quadratic rather than linear dependence on the transcription-factor
concentrations.

A particularly close connection is provided by the ``exclusive
switch'' model of Lipshtat et al.~\cite{Lipshtat2006}.  With monomeric
binding, its deterministic rate equations can be written
\[
 \dot x_1
 =r\,\frac{1+kx_1}{1+kx_1+kx_2}-d x_1,
 \qquad
 \dot x_2
 =r\,\frac{1+kx_2}{1+kx_1+kx_2}-d x_2.
\]
Thus this model is the symmetric special case
\begin{gather*}
 K_1=K_2=K_{12}=K_{21}=1/k,\\
 \alpha_{01}=\alpha_{02}=\alpha_1=\alpha_2=r,\qquad
 \delta_1=\delta_2=d
\end{gather*}
of \eqref{eq:system1}--\eqref{eq:system2}.  Importantly, however, this
deterministic special case has exactly one positive equilibrium for
every choice of parameters.  The bistability found for the exclusive
switch without cooperative binding is a stochastic effect, not
deterministic bistability of these rate equations; see
\cite{Lipshtat2006,Loinger2007} and
Remark~\ref{rem:exclusive-monostable} below.

The well-known synthetic genetic toggle switch of Gardner, Cantor, and
Collins~\cite{GardnerCantorCollins2000} is related in spirit but is
described by a different class of equations.  Its standard
deterministic model uses mutual-repression Hill functions, for example
\[
 \dot u=\frac{\rho_1}{1+v^\beta}-u,
 \qquad
 \dot v=\frac{\rho_2}{1+u^\eta}-v,
\]
where the Hill exponents represent cooperative repression.  Its
production terms therefore have neither the common competitive
denominators nor the self-regulatory numerator terms that are central
to \eqref{eq:system1}--\eqref{eq:system2}.  Thus, although that work is
an important biological precedent for bistable two-gene toggle
switches, its mathematical model is not a special case of the system
studied here.

\section{Nullclines and the number of positive equilibria}

Write the vector field as
\[
\dot x_1=f(x_1,x_2),\qquad
\dot x_2=g(x_1,x_2).
\]

\begin{lemma}\label{lem:nullclines}
Define, for $x_1>0$,
\begin{equation}\label{eq:h-def}
h(x_1):=
K_{21}\left(
\frac{\alpha_{01}}{\delta_1 x_1}
+\frac{\alpha_1}{\delta_1 K_1}
-1-\frac{x_1}{K_1}
\right).
\end{equation}
Then $h$ is $C^\infty$ and strictly decreasing,
\begin{equation}\label{eq:h-prime}
h'(x_1)
=-K_{21}\left(
\frac{\alpha_{01}}{\delta_1 x_1^2}+\frac{1}{K_1}
\right)<0,
\end{equation}
and
\[
\lim_{x_1\downarrow0}h(x_1)=+\infty,
\qquad
\lim_{x_1\to\infty}h(x_1)=-\infty.
\]
Consequently, $h$ has a unique zero $a>0$, and the positive part of
the first nullcline is precisely the graph
\[
x_2=h(x_1),\qquad 0<x_1<a.
\]

Similarly, define, for $x_2>0$,
\begin{equation}\label{eq:k-def}
k(x_2):=
K_{12}\left(
\frac{\alpha_{02}}{\delta_2 x_2}
+\frac{\alpha_2}{\delta_2 K_2}
-1-\frac{x_2}{K_2}
\right).
\end{equation}
Then $k$ is $C^\infty$ and strictly decreasing, has a unique zero
$b>0$, and the positive part of the second nullcline is
\[
x_1=k(x_2),\qquad 0<x_2<b.
\]
Moreover, $k:(0,b)\to(0,\infty)$ is a $C^\infty$ bijection with
nonvanishing derivative.  Its inverse
\begin{equation}\label{eq:ell-def}
\ell:(0,\infty)\longrightarrow(0,b)
\end{equation}
is therefore $C^\infty$ and strictly decreasing.

The positive equilibria are in one-to-one correspondence with the
zeros in $(0,a)$ of
\begin{equation}\label{eq:H-def}
H(x_1):=h(x_1)-\ell(x_1).
\end{equation}
Furthermore,
\begin{equation}\label{eq:H-end-signs}
\lim_{x_1\downarrow0}H(x_1)=+\infty,
\qquad
\lim_{x_1\uparrow a}H(x_1)=-\ell(a)<0.
\end{equation}
In particular, the system has at least one positive equilibrium.
Distinct positive equilibria have distinct $x_1$-coordinates; if they
are ordered by increasing $x_1$, then their $x_2$-coordinates are
ordered in the opposite direction.
\end{lemma}

\begin{proof}
Solving $f(x_1,x_2)=0$ for $x_2$ gives
$x_2=h(x_1)$, and differentiation gives \eqref{eq:h-prime}.  The
limits and strict monotonicity of $h$ give the unique zero $a$ and the
description of the positive first nullcline.

The same calculation for $g(x_1,x_2)=0$ gives
$x_1=k(x_2)$ and
\[
k'(x_2)
=-K_{12}\left(
\frac{\alpha_{02}}{\delta_2 x_2^2}+\frac{1}{K_2}
\right)<0.
\]
Also $k(x_2)\to+\infty$ as $x_2\downarrow0$, while
$k(x_2)\to0$ as $x_2\uparrow b$.  Thus
$k:(0,b)\to(0,\infty)$ is a bijection.  Since $k'$ never vanishes,
the inverse function theorem implies that its inverse $\ell$ is
$C^\infty$; it is strictly decreasing because $k$ is.

The equilibrium condition is exactly $h(x_1)=\ell(x_1)$.  As
$x_1\downarrow0$, $h(x_1)\to+\infty$ whereas
$\ell(x_1)\to b$; as $x_1\uparrow a$, $h(x_1)\to0$ whereas
$\ell(a)>0$.  This proves \eqref{eq:H-end-signs}, and existence follows
from the intermediate value theorem.

Finally, for any fixed $x_1>0$, the first nullcline contains at most
one point, namely $(x_1,h(x_1))$.  Hence two distinct positive
equilibria cannot have the same $x_1$-coordinate.  The reverse ordering
of their $x_2$-coordinates follows from $h'<0$.
\end{proof}

\begin{lemma}\label{lem:polynomial}
There is a nonzero polynomial $P$ of degree at most four such that the
positive equilibria are precisely the points $(x_1,h(x_1))$ for which
$0<x_1<a$ and $P(x_1)=0$.

More precisely, if $x_1^*$ is the first coordinate of a positive
equilibrium and has multiplicity $m$ as a root of $P$, then, in a
neighborhood of $x_1^*$,
\begin{equation}\label{eq:H-order}
H(x_1)=(x_1-x_1^*)^m R(x_1),
\end{equation}
where $R$ is $C^\infty$ and $R(x_1^*)\ne0$.  Thus $m$ is the order of
vanishing of $H$ at $x_1^*$, and $H$ changes sign at $x_1^*$ if and
only if $m$ is odd.
\end{lemma}

\begin{proof}
Clearing the positive denominator in $f=0$ gives
\begin{equation}\label{eq:F}
A x_1^2+B x_1x_2+C x_1-\alpha_{01}=0,
\end{equation}
where
\begin{equation}\label{eq:ABC}
A=\frac{\delta_1}{K_1}>0,
\qquad
B=\frac{\delta_1}{K_{21}}>0,
\qquad
C=\delta_1-\frac{\alpha_1}{K_1}.
\end{equation}
Similarly, $g=0$ is equivalent to
\begin{equation}\label{eq:G}
G(x_1,x_2):=
D x_2^2+E x_1x_2+F x_2-\alpha_{02}=0,
\end{equation}
where
\begin{equation}\label{eq:DEF}
D=\frac{\delta_2}{K_2}>0,
\qquad
E=\frac{\delta_2}{K_{12}}>0,
\qquad
F=\delta_2-\frac{\alpha_2}{K_2}.
\end{equation}
From \eqref{eq:F},
\begin{equation}\label{eq:N}
x_2=h(x_1)=\frac{N(x_1)}{B x_1},
\qquad
N(x_1):=\alpha_{01}-Ax_1^2-Cx_1.
\end{equation}
Substitution into \eqref{eq:G}, followed by multiplication by
$B^2x_1^2$, gives
\begin{equation}\label{eq:P}
P(x_1)
:=D N(x_1)^2
+EBx_1^2N(x_1)
+FBx_1N(x_1)
-\alpha_{02}B^2x_1^2.
\end{equation}
Thus $\deg P\leq4$, and
\begin{equation}\label{eq:P-at-zero}
P(0)=D\alpha_{01}^2>0,
\end{equation}
so $P$ is not identically zero.  By construction,
\begin{equation}\label{eq:P-G}
P(x_1)=B^2x_1^2G(x_1,h(x_1)).
\end{equation}
Hence, for $0<x_1<a$, the zeros of $P$ correspond exactly to the
positive equilibria.

It remains to compare orders of vanishing.  Since
\[
G_{x_1}(x_1,x_2)=E x_2>0
\]
in the positive quadrant, differentiation of
$G(k(x_2),x_2)=0$ gives
\[
G_{x_2}(k(x_2),x_2)
=-G_{x_1}(k(x_2),x_2)k'(x_2)>0.
\]
Because $G(x_1,\ell(x_1))=0$ and $G$ is quadratic in $x_2$,
\begin{align}
G(x_1,h(x_1))
&=G(x_1,h(x_1))-G(x_1,\ell(x_1)) \notag\\
&=\bigl(h(x_1)-\ell(x_1)\bigr)
\bigl[D(h(x_1)+\ell(x_1))+Ex_1+F\bigr] \notag\\
&=H(x_1)Q(x_1),
\label{eq:factorization}
\end{align}
where
\[
Q(x_1):=D(h(x_1)+\ell(x_1))+Ex_1+F.
\]
Consequently, on all of $(0,\infty)$,
\begin{equation}\label{eq:P-factorization}
P(x_1)=B^2x_1^2Q(x_1)H(x_1).
\end{equation}
At a positive equilibrium $(x_1^*,x_2^*)$,
\[
Q(x_1^*)=2Dx_2^*+Ex_1^*+F
=G_{x_2}(x_1^*,x_2^*)>0.
\]
Thus $Q$ is nonzero in a neighborhood of $x_1^*$.  If $m$ is the
multiplicity of $x_1^*$ as a root of $P$, divide
\eqref{eq:P-factorization} by
$(x_1-x_1^*)^mB^2x_1^2Q(x_1)$ to obtain
\eqref{eq:H-order} with $R(x_1^*)\ne0$.  The final sign-change
statement follows immediately.
\end{proof}

\begin{lemma}\label{lem:jacobian}
Along the positive part of the first nullcline,
$f_{x_1}<0$, and along the positive part of the second nullcline,
$g_{x_2}<0$.  In particular, at every positive equilibrium,
\begin{equation}\label{eq:trace-negative}
\operatorname{tr}J<0,
\end{equation}
where $J$ is the Jacobian matrix of the vector field.

Moreover, at every positive equilibrium $(x_1^*,x_2^*)$,
\begin{equation}\label{eq:Hprime-det}
H'(x_1^*)
=-\frac{\det J(x_1^*,x_2^*)}
{f_{x_2}(x_1^*,x_2^*)\,g_{x_2}(x_1^*,x_2^*)},
\end{equation}
and the denominator is positive.  Consequently, a positive equilibrium
is degenerate if and only if the corresponding zero of $H$ has order
of vanishing greater than one.
\end{lemma}

\begin{proof}
Direct differentiation gives
\begin{equation}\label{eq:offdiagonal}
f_{x_2}<0,
\qquad
g_{x_1}<0
\end{equation}
throughout the positive quadrant.  Differentiating
$f(x_1,h(x_1))=0$ along the positive first nullcline gives
\[
f_{x_1}+f_{x_2}h'=0,
\qquad\text{hence}\qquad
h'=-\frac{f_{x_1}}{f_{x_2}}.
\]
Since $h'<0$ and $f_{x_2}<0$, it follows that
\begin{equation}\label{eq:f11-negative}
f_{x_1}<0
\end{equation}
along the entire positive first nullcline.

Similarly, differentiating $g(k(x_2),x_2)=0$ along the positive second
nullcline gives
\[
g_{x_1}k'(x_2)+g_{x_2}=0.
\]
Since $g_{x_1}<0$ and $k'<0$, we obtain
\begin{equation}\label{eq:g22-negative}
g_{x_2}<0
\end{equation}
along the entire positive second nullcline.  This proves
\eqref{eq:trace-negative} at an equilibrium.

Because $x_2=\ell(x_1)$ is the inverse representation of the second
nullcline,
\[
\ell'=-\frac{g_{x_1}}{g_{x_2}}.
\]
Therefore, at an equilibrium,
\begin{align*}
H'
&=h'-\ell'\\
&=-\frac{f_{x_1}}{f_{x_2}}+\frac{g_{x_1}}{g_{x_2}}\\
&=-\frac{f_{x_1}g_{x_2}-f_{x_2}g_{x_1}}
{f_{x_2}g_{x_2}}
=-\frac{\det J}{f_{x_2}g_{x_2}}.
\end{align*}
By \eqref{eq:offdiagonal} and \eqref{eq:g22-negative}, the denominator
is positive.  Hence $\det J=0$ if and only if $H'=0$, and
Lemma~\ref{lem:polynomial} shows that this is equivalent to order of
vanishing greater than one.
\end{proof}

\begin{proposition}\label{prop:number}
The system \eqref{eq:system1}--\eqref{eq:system2} has either one, two,
or three distinct equilibria in $\mathbb R_{>0}^2$.

If it has exactly two distinct positive equilibria, then one is
nondegenerate and the other is degenerate; equivalently, one
nullcline intersection is transverse and the other is a tangency.  If
it has exactly three distinct positive equilibria, then all three are
nondegenerate.  Consequently, any configuration in which all positive
nullcline intersections are transverse has either one or three
positive equilibria.
\end{proposition}

\begin{proof}
Existence of at least one positive equilibrium follows from
Lemma~\ref{lem:nullclines}.  By Lemma~\ref{lem:polynomial}, every zero
of $H$ in $(0,a)$ is a root of a nonzero polynomial of degree at most
four, and its order of vanishing is the corresponding root
multiplicity.  Hence five or more distinct zeros are impossible.

There cannot be four distinct positive equilibria.  If $H$ had four
distinct zeros in $(0,a)$, then all four would be simple because their
total multiplicity is at most four.  By Lemma~\ref{lem:polynomial},
$H$ would change sign at each of them.  Its signs near the two
endpoints of $(0,a)$ would therefore be the same, contradicting
\eqref{eq:H-end-signs}.  Hence there are at most three distinct
positive equilibria.

Suppose that there are exactly two, with orders of vanishing $m_1$ and
$m_2$.  Since the endpoint signs in \eqref{eq:H-end-signs} are
opposite, an odd number of these two zeros must have odd order.
Therefore one order is odd and the other is even.  Since
\[
m_1+m_2\leq4,
\]
the only possibility, up to ordering, is
\[
(m_1,m_2)=(1,2).
\]
Thus one zero is simple and the other is double.  By
Lemma~\ref{lem:jacobian}, the corresponding equilibria are,
respectively, nondegenerate and degenerate.  Since the two nullclines
are smooth graphs, $H'=0$ at the latter point is exactly equality of
their slopes, that is, tangency.

Finally, suppose that there are exactly three distinct positive
equilibria, with orders of vanishing $m_1,m_2,m_3$.  Their total is at
least three and at most four.  If it were four, exactly one $m_i$ would
equal two and the other two would equal one.  There would then be
exactly two odd-order zeros, so the sign of $H$ would be the same near
the two endpoints, again contradicting \eqref{eq:H-end-signs}.
Therefore
$
m_1=m_2=m_3=1
$.
Lemma~\ref{lem:jacobian} now implies that all three equilibria are
nondegenerate.
\end{proof}

\begin{remark}[The three-equilibrium hypothesis is nonvacuous]
\label{rem:examples}
Here are two symmetric examples.  First, take
\[
\alpha_{01}=\alpha_{02}=0.1,\qquad
\alpha_1=\alpha_2=10,\qquad
\delta_1=\delta_2=1,
\]
\[
K_1=K_2=1,\qquad K_{12}=K_{21}=0.2.
\]
For these parameters the polynomial \eqref{eq:P} factors, up to a
nonzero constant, as
\[
(40x_1^2-360x_1+1)(60x_1^2-90x_1-1),
\]
and has three positive roots corresponding to the three positive
equilibria
\[
(0.00278,8.99722),\;
(1.51103,1.51103),\;
(8.99722,0.00278),
\]
where the displayed values are rounded.

As an example with no change in production upon self-binding, take
\[
\alpha_{01}=\alpha_{02}=\alpha_1=\alpha_2=100,\qquad
\delta_1=\delta_2=1,
\]
\[
K_1=K_2=1,\qquad K_{12}=K_{21}=0.01.
\]
Now \eqref{eq:P} factors, again up to a nonzero constant, as
\[
(99x_1^2-9801x_1+100)(101x_1^2-99x_1-100),
\]
and the three positive equilibria are approximately
\begin{gather*}
(0.01020,98.98980),\qquad
(1.59929,1.59929),\\
(98.98980,0.01020).
\end{gather*}
Thus Hill exponent one does not by itself preclude deterministic
bistability in the present class.
\end{remark}

\section{Proof of Theorem~\ref{thm:bistability}}

We use the following standard planar convergence result for competitive
systems; see~\cite{Hirsch1982}.

\begin{lemma}[Planar competitive convergence]\label{lem:competitive-convergence}
Let $U\subset\mathbb R^2$ be an open rectangle and consider a $C^1$
system
\[
\dot z_1=F_1(z_1,z_2),\qquad
\dot z_2=F_2(z_1,z_2)
\]
on $U$.  Assume that
\[
\frac{\partial F_1}{\partial z_2}<0,
\qquad
\frac{\partial F_2}{\partial z_1}<0
\]
throughout $U$.  If the positive semiorbit of a solution has compact
closure contained in $U$, then that solution converges, as
$t\to+\infty$, to an equilibrium.
\end{lemma}

\begin{proof}[Proof of Theorem~\ref{thm:bistability}]
By Lemma~\ref{lem:nullclines}, the three distinct equilibria can be
labeled so that \eqref{eq:ordering} holds.  By
Proposition~\ref{prop:number}, all three are nondegenerate, and the
corresponding zeros
\[
x_1^{(1)}<x_1^{(2)}<x_1^{(3)}
\]
of $H$ are simple.  Since $H$ is positive near $0$, negative near $a$,
and changes sign at each simple zero, we have
\begin{equation}\label{eq:Hprime-signs}
H'(x_1^{(1)})<0,
\qquad
H'(x_1^{(2)})>0,
\qquad
H'(x_1^{(3)})<0.
\end{equation}
By Lemma~\ref{lem:jacobian},
\begin{equation}\label{eq:det-pattern}
\det J(p^{(1)})>0,
\qquad
\det J(p^{(2)})<0,
\qquad
\det J(p^{(3)})>0,
\end{equation}
and $\operatorname{tr}J<0$ at all three equilibria.

At every positive equilibrium,
\begin{align}
(\operatorname{tr}J)^2-4\det J
&=(f_{x_1}+g_{x_2})^2
-4(f_{x_1}g_{x_2}-f_{x_2}g_{x_1}) \notag\\
&=(f_{x_1}-g_{x_2})^2+4f_{x_2}g_{x_1}>0,
\label{eq:discriminant}
\end{align}
because $f_{x_2}<0$ and $g_{x_1}<0$.  Thus the eigenvalues at
$p^{(1)}$ and $p^{(3)}$ are real and negative, whereas those at
$p^{(2)}$ are real and have opposite signs.  Hence $p^{(1)}$ and
$p^{(3)}$ are hyperbolic asymptotically stable nodes and $p^{(2)}$ is a
hyperbolic saddle.  These conclusions use the standard linearization
principle; see, for example,~\cite{mct}.

It remains to prove the global assertion.  The vector field extends
smoothly to an open neighborhood of the closed positive quadrant,
because both denominators are strictly positive on that quadrant.
We first prove that $\mathbb R_{>0}^2$ is forward invariant.  On the
face $x_1=0$, $x_2\geq0$,
\[
f(0,x_2)=\frac{\alpha_{01}}{1+x_2/K_{21}}>0,
\]
and on the face $x_2=0$, $x_1\geq0$,
\[
g(x_1,0)=\frac{\alpha_{02}}{1+x_1/K_{12}}>0.
\]
Suppose, for contradiction, that a solution with
$x_1(0),x_2(0)>0$ has a finite first exit time $\tau$ from the open
positive quadrant.  By continuity,
$x_1(\tau),x_2(\tau)\geq0$, and at least one coordinate is zero.  If
$x_1(\tau)=0$, then $x_1(t)>0$ for $0\leq t<\tau$, so
\[
\dot x_1(\tau)
=\lim_{s\downarrow0}\frac{x_1(\tau)-x_1(\tau-s)}{s}\leq0.
\]
But the vector field gives $\dot x_1(\tau)=f(0,x_2(\tau))>0$, a
contradiction.  The case $x_2(\tau)=0$ is identical.  Thus the open
positive quadrant is forward invariant on the maximal forward interval
of existence.

Set
\[
M_1:=\max\{\alpha_{01},\alpha_1\},
\qquad
M_2:=\max\{\alpha_{02},\alpha_2\}.
\]
For $x_1,x_2\geq0$,
\begin{align*}
\frac{\alpha_{01}+\alpha_1x_1/K_1}
{1+x_1/K_1+x_2/K_{21}}
&\leq
\frac{\alpha_{01}+\alpha_1x_1/K_1}{1+x_1/K_1}
\leq M_1,\\
\frac{\alpha_{02}+\alpha_2x_2/K_2}
{1+x_2/K_2+x_1/K_{12}}
&\leq
\frac{\alpha_{02}+\alpha_2x_2/K_2}{1+x_2/K_2}
\leq M_2.
\end{align*}
Consequently,
\[
\dot x_1\leq M_1-\delta_1 x_1,
\qquad
\dot x_2\leq M_2-\delta_2 x_2.
\]
For a solution starting at
$x^0=(x_1^0,x_2^0)\in\mathbb R_{>0}^2$, scalar comparison gives
\begin{equation}\label{eq:upper-bounds}
0<x_1(t)\leq U_1,
\qquad
0<x_2(t)\leq U_2,
\qquad t\geq0,
\end{equation}
where
\[
U_1:=\max\left\{x_1^0,\frac{M_1}{\delta_1}\right\},
\qquad
U_2:=\max\left\{x_2^0,\frac{M_2}{\delta_2}\right\}.
\]
In particular, the solution cannot escape to infinity in finite time,
so it is defined for all $t\geq0$.

The same bounds keep the orbit uniformly away from the coordinate
axes.  Define
\begin{align*}
c_1&:=\frac{\alpha_{01}}
{1+U_1/K_1+U_2/K_{21}}>0,\\
c_2&:=\frac{\alpha_{02}}
{1+U_2/K_2+U_1/K_{12}}>0.
\end{align*}
Along the solution,
\[
\dot x_1\geq c_1-\delta_1x_1,
\qquad
\dot x_2\geq c_2-\delta_2x_2.
\]
A second comparison gives
\begin{align}
x_1(t)
&\geq \frac{c_1}{\delta_1}
+\left(x_1^0-\frac{c_1}{\delta_1}\right)e^{-\delta_1t}
\geq \min\left\{x_1^0,\frac{c_1}{\delta_1}\right\}>0,
\label{eq:lower-bound1}\\
x_2(t)
&\geq \frac{c_2}{\delta_2}
+\left(x_2^0-\frac{c_2}{\delta_2}\right)e^{-\delta_2t}
\geq \min\left\{x_2^0,\frac{c_2}{\delta_2}\right\}>0.
\label{eq:lower-bound2}
\end{align}
Thus every positive semiorbit has compact closure contained in
$\mathbb R_{>0}^2$.

By \eqref{eq:offdiagonal}, the system is strongly competitive on
$\mathbb R_{>0}^2$.  Lemma~\ref{lem:competitive-convergence} therefore
implies that every positive solution converges to an equilibrium.  By
hypothesis the only positive equilibria are $p^{(1)},p^{(2)},p^{(3)}$,
so every positive solution converges to exactly one of these three
points.  This proves the disjoint decomposition
\eqref{eq:global-partition}.

Since $p^{(1)}$ and $p^{(3)}$ are locally asymptotically stable, their
basins $\mathcal B_1$ and $\mathcal B_3$ are nonempty and open.  Since
$p^{(2)}$ is a hyperbolic saddle, the stable manifold theorem implies
that its stable set $W^s(p^{(2)})$ is a one-dimensional immersed stable
manifold~\cite{Perko2001}.  Therefore every initial condition outside
$W^s(p^{(2)})$ converges to one of the two locally asymptotically stable
equilibria.  This proves the global assertion.
\end{proof}

\begin{remark}[The symmetric exclusive switch is monostable]
\label{rem:exclusive-monostable}
Consider the symmetric special case
\begin{gather*}
\alpha_{01}=\alpha_{02}=\alpha_1=\alpha_2=r,\\
K_1=K_2=K_{12}=K_{21}=1/k,\qquad
\delta_1=\delta_2=d,
\end{gather*}
which gives the deterministic exclusive-switch equations of
\cite{Lipshtat2006,Loinger2007}.  This system has exactly one positive
equilibrium for every $r,k,d>0$, in agreement with the direct
rate-equation analysis in~\cite{Loinger2007}.

For completeness, this also follows from the results above.  Let
\[
S_1=1+x_1/K_1+x_2/K_{21}.
\]
On the first nullcline,
\[
\alpha_{01}+\alpha_1x_1/K_1=\delta_1x_1S_1,
\]
and direct differentiation, followed by use of this identity, gives
\[
f_{x_1}
=-\frac{\alpha_{01}/x_1+\delta_1x_1/K_1}{S_1},
\qquad
f_{x_2}
=-\frac{\delta_1x_1}{K_{21}S_1}.
\]
The map $(x_1,x_2)\mapsto(x_2,x_1)$ maps equilibria to equilibria.
Moreover, $h(s)-s$ is strictly decreasing from $+\infty$ to
$-\infty$, so there is exactly one symmetric equilibrium $(u,u)$.
Any nonsymmetric equilibrium must occur together with its reflected
partner, so two equilibria are impossible.

If there were three equilibria, the symmetric one would be the middle
equilibrium in the ordering \eqref{eq:ordering}, and
Theorem~\ref{thm:bistability} would imply that it is a saddle.  At
$(u,u)$, however, the Jacobian has the form
\[
J=\begin{pmatrix}a&b\\ b&a\end{pmatrix},
\]
where $a<0$, $b<0$, and
\[
|a|-|b|=\frac{r}{uS_1}>0.
\]
Hence $\det J=a^2-b^2>0$, a contradiction.  Thus the symmetric
equilibrium is the unique positive equilibrium.  In this special case,
the bistability reported without cooperative binding is therefore
stochastic rather than deterministic.
\end{remark}

\begin{remark}[Transversality]
Proposition~\ref{prop:number} makes precise the sense in which the
two-equilibrium case is exceptional.  By a \emph{transverse
configuration} we mean one for which the two positive nullclines meet
transversely at every positive equilibrium, equivalently
$\det J\ne0$ at every positive equilibrium.  Proposition~\ref{prop:number}
shows that a transverse configuration has either one or three positive
equilibria.  No parameter-space genericity assertion is needed for the
results of this paper.
\end{remark}

\begin{remark}[Index-theoretic viewpoint]
Once nondegeneracy has been established, planar index theory gives an
alternative way to account for the equilibrium types.  The rectangle
used for the index calculation can be chosen explicitly in the proper
order.  First choose
\[
R_1>\frac{M_1}{\delta_1},
\qquad
R_2>\frac{M_2}{\delta_2}.
\]
Then the vector field points to the left on $x_1=R_1$ and downward on
$x_2=R_2$.  Next choose $\varepsilon>0$ sufficiently small that
\begin{gather*}
\frac{\alpha_{01}}
 {1+\varepsilon/K_1+R_2/K_{21}}-\delta_1\varepsilon>0,\\
\frac{\alpha_{02}}
 {1+\varepsilon/K_2+R_1/K_{12}}-\delta_2\varepsilon>0.
\end{gather*}
It follows that the vector field points to the right on
$x_1=\varepsilon$ and upward on $x_2=\varepsilon$.  The same
inequalities show that no equilibrium can lie between either
coordinate axis and the corresponding near side, while the upper
bounds used in the proof of Theorem~\ref{thm:bistability} show that no
equilibrium can lie beyond the right or top side.  Thus the rectangle
\[
[\varepsilon,R_1]\times[\varepsilon,R_2]
\]
contains all positive equilibria and the vector field points strictly
inward along its boundary.

The index of the vector field along this boundary is $+1$.  Hence the
sum of the indices of the equilibria inside the rectangle is $+1$.
For a nondegenerate planar equilibrium the index is
$\operatorname{sgn}(\det J)$, so three nondegenerate equilibria must
consist of two equilibria of index $+1$ and one of index $-1$.  Since
\eqref{eq:trace-negative} holds and \eqref{eq:discriminant} rules out
foci, the index-$+1$ equilibria are stable nodes and the index-$-1$
equilibrium is a saddle.  The nullcline argument additionally
identifies the middle equilibrium as the saddle.  See, for example,
\cite{Perko2001} for the planar index theory used here.
\end{remark}

\begin{remark}
The adjective ``global'' refers to the convergence and basin
decomposition \eqref{eq:global-partition}: every positive trajectory
converges to one of the three equilibria, and the only positive initial
conditions that fail to converge to one of the two sinks are those on
the stable manifold of the saddle.  It does not mean that either stable
equilibrium is itself globally asymptotically stable on the entire
positive quadrant.
\end{remark}

\section{The general case}\label{sec:general}

The proof of Theorem~\ref{thm:bistability} used the specific form of
\eqref{eq:system1}--\eqref{eq:system2} in two different ways.  On the
one hand, it used several qualitative features: strong competition, the
fact that both nullclines are decreasing graphs, and the endpoint signs
\eqref{eq:H-end-signs}.  On the other hand, through
Lemma~\ref{lem:polynomial}, it used the algebraic fact that the orders
of vanishing of $H$ are root multiplicities of a polynomial of degree at
most four.  It is natural to ask whether the qualitative features alone
already force the bistable phase portrait whenever there are exactly
three positive equilibria.

The answer is no, and the obstruction comes from degenerate equilibria.
We first isolate the qualitative structure and show that, under it,
three positive equilibria lead to one of two alternatives
(Theorem~\ref{thm:general}): either $H$ changes sign at all three
equilibria, in which case the configuration of
Theorem~\ref{thm:bistability} follows whenever the equilibria are
nondegenerate, or $H$ changes sign at only one of them, in which case at
least two of the three equilibria are degenerate.  We then give an explicit polynomial
example in which the second alternative occurs and the system is not
bistable (Proposition~\ref{prop:example-stability}).  Finally, we return
to \eqref{eq:system1}--\eqref{eq:system2} and explain how
Lemma~\ref{lem:polynomial} excludes the second alternative
(Proposition~\ref{prop:special-case}).

\subsection{Qualitative hypotheses}

In this section, $f$ and $g$ denote arbitrary $C^1$ functions on
$\mathbb R^2_{>0}$, and we consider the system
\[
\dot x_1=f(x_1,x_2),\qquad
\dot x_2=g(x_1,x_2).
\]
As before, the first and second nullclines are the zero sets of $f$ and
$g$, their positive parts are their intersections with
$\mathbb R^2_{>0}$, and $J$ denotes the Jacobian matrix of the vector
field.  An equilibrium is \emph{degenerate} if $\det J=0$ there, and
\emph{nondegenerate} otherwise.  We consider the following hypotheses.

\begin{enumerate}
\item[(A1)] \emph{Strong competition.}  $f_{x_2}<0$ and $g_{x_1}<0$
throughout $\mathbb R^2_{>0}$.

\item[(A2)] \emph{Decreasing first nullcline.}  There are $a>0$ and a
$C^1$ function $h:(0,a)\to(0,\infty)$ with $h'<0$ such that the positive
part of the first nullcline is precisely the graph
\[
x_2=h(x_1),\qquad 0<x_1<a.
\]

\item[(A3)] \emph{Decreasing second nullcline.}  There are $b>0$ and a
$C^1$ function $k:(0,b)\to(0,\infty)$ with $k'<0$ such that the
positive part of the second nullcline is precisely the graph
\[
x_1=k(x_2),\qquad 0<x_2<b,
\]
and such that the image $k((0,b))$ contains the interval $(0,a)$.
\end{enumerate}

For \eqref{eq:system1}--\eqref{eq:system2}, the image $k((0,b))$ is all
of $(0,\infty)$ by Lemma~\ref{lem:nullclines}, since $\alpha_{02}>0$
forces $k(x_2)\to+\infty$ as $x_2\downarrow0$.  The containment
condition in (A3) is a weakening of this property; in the example of
Section~\ref{sec:example} it holds with image $(0,22)$.

Under (A3), $k$ is a strictly decreasing $C^1$ bijection from $(0,b)$
onto its image, with nonvanishing derivative.  We denote by
\[
\ell:(0,a)\longrightarrow(0,b)
\]
the restriction to $(0,a)$ of its inverse.  By the inverse function
theorem, $\ell$ is $C^1$ and strictly decreasing, with
$\ell'=1/(k'\circ\ell)$.  As in Section~2, we set
\begin{equation}\label{eq:H-general}
H(x_1):=h(x_1)-\ell(x_1),\qquad 0<x_1<a.
\end{equation}
We use the same symbols as in Lemma~\ref{lem:nullclines} on purpose: for
\eqref{eq:system1}--\eqref{eq:system2}, the functions $h$ and $k$ are
those of Lemma~\ref{lem:nullclines}, and $\ell$ and $H$ are the
restrictions to $(0,a)$ of \eqref{eq:ell-def} and \eqref{eq:H-def}; see
Proposition~\ref{prop:special-case} below.  The two remaining hypotheses
are the following.

\begin{enumerate}
\item[(A4)] \emph{Endpoint signs.}  There is $\theta\in(0,a/2)$ such that
$H>0$ on $(0,\theta)$ and $H<0$ on $(a-\theta,a)$.

\item[(A5)] \emph{Precompactness.}  Every solution with initial condition
in $\mathbb R^2_{>0}$ is defined for all $t\geq0$, and its positive
semiorbit has compact closure contained in $\mathbb R^2_{>0}$.
\end{enumerate}

Hypotheses (A1)--(A4) are used for the local conclusions, and (A5) only
for the global one.  The first step is to observe that the arguments of
Lemmas~\ref{lem:nullclines} and~\ref{lem:jacobian} used only these
qualitative features.

\begin{lemma}\label{lem:general-structure}
Assume (A1)--(A3).
\begin{enumerate}
\item[(a)] The positive equilibria are in one-to-one correspondence with
the zeros of $H$ in $(0,a)$, via $x_1\mapsto(x_1,h(x_1))$.  Distinct
positive equilibria have distinct $x_1$-coordinates; if they are ordered
by increasing $x_1$, then their $x_2$-coordinates are ordered in the
opposite direction.

\item[(b)] $f_{x_1}<0$ along the positive part of the first nullcline,
and $g_{x_2}<0$ along the positive part of the second nullcline.
Consequently, at every positive equilibrium,
\[
\operatorname{tr}J<0
\qquad\text{and}\qquad
(\operatorname{tr}J)^2-4\det J>0 .
\]

\item[(c)] At every positive equilibrium $(x_1^*,x_2^*)$, identity
\eqref{eq:Hprime-det} holds, and its denominator is positive.
Consequently, a positive equilibrium is degenerate if and only if
$H'(x_1^*)=0$, and in that case the eigenvalues of $J$ are $0$ and
$\operatorname{tr}J<0$.
\end{enumerate}
\end{lemma}

\begin{proof}
(a) Let $x\in\mathbb R^2_{>0}$.  By (A2), $f(x)=0$ if and only if
$0<x_1<a$ and $x_2=h(x_1)$.  By (A3), $g(x)=0$ if and only if
$0<x_2<b$ and $x_1=k(x_2)$.  If $0<x_1<a$, then $x_1$ lies in the image
of $k$, and since $k$ is injective, the condition $x_1=k(x_2)$ with
$0<x_2<b$ is equivalent to $x_2=\ell(x_1)$.  Hence $x$ is an equilibrium
if and only if $0<x_1<a$ and $x_2=h(x_1)=\ell(x_1)$, that is, if and only
if $x=(x_1,h(x_1))$ with $H(x_1)=0$.  The statements about the
coordinates follow from $h'<0$.

(b) Differentiating $f(x_1,h(x_1))=0$ gives $f_{x_1}+f_{x_2}h'=0$ along
the positive first nullcline, so $f_{x_1}=-f_{x_2}h'<0$ by (A1) and
(A2).  Differentiating $g(k(x_2),x_2)=0$ gives $g_{x_1}k'+g_{x_2}=0$
along the positive second nullcline, so $g_{x_2}=-g_{x_1}k'<0$ by (A1)
and (A3).  At a positive equilibrium both conclusions apply, so
$\operatorname{tr}J<0$.  The computation \eqref{eq:discriminant} used
only $f_{x_2}<0$ and $g_{x_1}<0$, so it applies verbatim.

(c) At a positive equilibrium $x^*=(x_1^*,x_2^*)$ we have
$x_2^*=h(x_1^*)=\ell(x_1^*)$ and $x_1^*=k(x_2^*)$.  By the two identities
in the proof of (b),
\[
h'(x_1^*)=-\frac{f_{x_1}(x^*)}{f_{x_2}(x^*)},
\qquad
\ell'(x_1^*)=\frac{1}{k'(x_2^*)}=-\frac{g_{x_1}(x^*)}{g_{x_2}(x^*)} .
\]
Subtracting gives \eqref{eq:Hprime-det}, exactly as in the proof of
Lemma~\ref{lem:jacobian}.  The denominator $f_{x_2}g_{x_2}$ is positive
by (A1) and (b).  Hence $\det J(x^*)=0$ if and only if $H'(x_1^*)=0$.  In
that case the characteristic polynomial of $J(x^*)$ is
$\lambda(\lambda-\operatorname{tr}J)$, so the eigenvalues are $0$ and
$\operatorname{tr}J<0$.
\end{proof}

\subsection{A dichotomy for three equilibria}

\begin{theorem}\label{thm:general}
Assume (A1)--(A4), and suppose that the system has exactly three
distinct equilibria in $\mathbb R^2_{>0}$.  Label them
$p^{(i)}=(x_1^{(i)},x_2^{(i)})$, $i=1,2,3$, so that
$x_1^{(1)}<x_1^{(2)}<x_1^{(3)}$.  Then
$x_2^{(1)}>x_2^{(2)}>x_2^{(3)}$, and exactly one of the following two
alternatives holds.
\begin{enumerate}
\item[(i)] $H$ changes sign at each of $x_1^{(1)}$, $x_1^{(2)}$,
$x_1^{(3)}$.  In this case
\[
\det J(p^{(1)})\geq0,
\qquad
\det J(p^{(2)})\leq0,
\qquad
\det J(p^{(3)})\geq0.
\]

\item[(ii)] $H$ changes sign, from positive to negative, at exactly one
of $x_1^{(1)}$, $x_1^{(2)}$, $x_1^{(3)}$.  At each of the other two,
$H'=0$ and $\det J=0$, and the eigenvalues of $J$ are $0$ and
$\operatorname{tr}J<0$.  In particular, at least two of the three
equilibria are degenerate, and neither of these two is hyperbolic.
\end{enumerate}
If all three equilibria are nondegenerate, then alternative~(i) holds
with strict inequalities, $p^{(1)}$ and $p^{(3)}$ are hyperbolic
asymptotically stable nodes, and $p^{(2)}$ is a hyperbolic saddle.

If (A5) also holds, then every solution with initial condition in
$\mathbb R^2_{>0}$ converges to one of the three equilibria.  If, in
addition, all three equilibria are nondegenerate, then the disjoint
decomposition \eqref{eq:global-partition} holds, the sets
$\mathcal B_1$ and $\mathcal B_3$ are nonempty open basins of
attraction, and $W^s(p^{(2)})$ is the one-dimensional stable manifold
of the saddle.
\end{theorem}

\begin{proof}
By Lemma~\ref{lem:general-structure}(a), the three equilibria have
distinct $x_1$-coordinates, these coordinates are the three zeros of $H$
in $(0,a)$, and the $x_2$-coordinates are ordered as stated.  Let
$x_1^{(0)}:=0$ and $x_1^{(4)}:=a$.  Since $H$ is continuous and has no
zero in the interval $I_j:=(x_1^{(j)},x_1^{(j+1)})$, it has a constant
sign $\sigma_j\in\{+1,-1\}$ on $I_j$, for $j=0,\dots,3$.  By (A4), $H$
has no zeros in $(0,\theta)$ or in $(a-\theta,a)$, so
$(0,\theta)\subset I_0$ and $(a-\theta,a)\subset I_3$.  Hence
$\sigma_0=+1$ and $\sigma_3=-1$.

Now, $H$ changes sign at $x_1^{(j)}$ if and only if
$\sigma_{j-1}\ne\sigma_j$.  Since
\[
(\sigma_0\sigma_1)(\sigma_1\sigma_2)(\sigma_2\sigma_3)=\sigma_0\sigma_3=-1,
\]
and a factor $\sigma_{j-1}\sigma_j$ equals $-1$ exactly when $H$ changes
sign at $x_1^{(j)}$, the number of zeros at which $H$ changes sign is
odd.  Thus $H$ changes sign either at all three zeros or at exactly one,
and the two alternatives are mutually exclusive.

We use two elementary facts.  If $H$ does not change sign at a zero
$x_1^*$, then $x_1^*$ is a local extremum of $H$, so $H'(x_1^*)=0$.  If
$H$ changes sign from positive to negative at $x_1^*$, then the
difference quotients $H(x_1)/(x_1-x_1^*)$ are negative for all $x_1\ne
x_1^*$ near $x_1^*$, so $H'(x_1^*)\leq0$; similarly,
$H'(x_1^*)\geq0$ at a change from negative to positive.

If $H$ changes sign at all three zeros, then
$(\sigma_0,\sigma_1,\sigma_2,\sigma_3)=(+1,-1,+1,-1)$, so
\[
H'(x_1^{(1)})\leq0,
\qquad
H'(x_1^{(2)})\geq0,
\qquad
H'(x_1^{(3)})\leq0.
\]
By Lemma~\ref{lem:general-structure}(c),
$\det J=-f_{x_2}g_{x_2}H'$ at each equilibrium, with
$f_{x_2}g_{x_2}>0$.  This gives the inequalities in~(i).  If $H$ changes
sign at exactly one zero, the change is from positive to negative,
because $\sigma_0=+1$ and $\sigma_3=-1$, and $H'=0$ at the other two
zeros.  The remaining assertions of~(ii) follow from
Lemma~\ref{lem:general-structure}(c).

Suppose now that all three equilibria are nondegenerate.  By
Lemma~\ref{lem:general-structure}(c), $H'\ne0$ at all three zeros, so
$H$ changes sign at each of them.  Thus alternative~(i) holds, and its
inequalities are strict.  By Lemma~\ref{lem:general-structure}(b), at
every positive equilibrium $\operatorname{tr}J<0$ and the eigenvalues
of $J$ are real.  At $p^{(1)}$ and $p^{(3)}$ we have $\det J>0$, so both
eigenvalues are negative; at $p^{(2)}$ we have $\det J<0$, so they have
opposite signs.  By the linearization principle~\cite{mct}, $p^{(1)}$
and $p^{(3)}$ are hyperbolic asymptotically stable nodes and $p^{(2)}$
is a hyperbolic saddle.

Finally, assume (A5).  By (A1), the system is strongly competitive on
the open rectangle $\mathbb R^2_{>0}$, so
Lemma~\ref{lem:competitive-convergence} shows that every positive
solution converges to an equilibrium, which is necessarily one of
$p^{(1)},p^{(2)},p^{(3)}$.  If all three equilibria are nondegenerate,
the remaining assertions follow exactly as in the last paragraph of the
proof of Theorem~\ref{thm:bistability}.
\end{proof}

\begin{remark}\label{rem:general-comments}
Note that alternative~(i) does not by itself
exclude degenerate equilibria.  If $H$ is smooth near a zero and
vanishes there to odd order at least three, then $H$ changes sign at
that zero and $H'=0$ there.  Theorem~\ref{thm:general} makes no
stability claim about such an equilibrium; a higher-order analysis would
be needed.  Observe that alternative~(ii) is not excluded by (A1)--(A4), as
the next example shows.
\end{remark}

\subsection{The degenerate alternative occurs}\label{sec:example}

We consider the following very simple example.  Both
nullclines are graphs of explicit polynomials, and the vector field is
built directly from them.

\begin{proposition}\label{prop:example}
Let
\[
h(x_1):=4-x_1,
\qquad
k(x_2):=4-x_2+\kappa(x_2),
\qquad
\kappa(y):=(2-y)(y-1)^2(y-3)^2,
\]
and consider the polynomial system
\begin{equation}\label{eq:example}
\dot x_1=f(x_1,x_2):=h(x_1)-x_2,
\qquad
\dot x_2=g(x_1,x_2):=k(x_2)-x_1,
\end{equation}
that is, $\dot x_1=4-x_1-x_2$ and $\dot x_2=4-x_1-x_2+\kappa(x_2)$.
\begin{enumerate}
\item[(a)] $\kappa'(y)\leq4/5$, and hence $k'(y)\leq-1/5$, for all real
$y$.

\item[(b)] Hypotheses (A1)--(A4) hold with $a=4$, with $h$ restricted to
$(0,4)$, and with $k$ restricted to $(0,b)$, where $b\in(3,7/2)$ is the
unique zero of $k$.  Moreover, $k((0,b))=(0,22)$.

\item[(c)] For $0<x_1<4$,
\begin{equation}\label{eq:example-H}
H(x_1)=\frac{\kappa(x_1)}{\mu(x_1)},
\qquad
\mu(x_1):=-\int_0^1k'\bigl(\ell(x_1)+sH(x_1)\bigr)\,ds\geq\frac15 ,
\end{equation}
and $\mu$ is $C^\infty$.  Consequently, the zeros of $H$ in $(0,4)$ are
$1$, $2$, and $3$, with orders of vanishing $2$, $1$, and $2$,
respectively; moreover $H>0$ on $(0,1)\cup(1,2)$ and $H<0$ on
$(2,3)\cup(3,4)$.

\item[(d)] The system has exactly three positive equilibria,
\[
p^{(1)}=(1,3),\qquad p^{(2)}=(2,2),\qquad p^{(3)}=(3,1).
\]
Alternative~(ii) of Theorem~\ref{thm:general} holds: $H$ changes sign
only at $x_1^{(2)}=2$, and $p^{(1)}$ and $p^{(3)}$ are degenerate.

\item[(e)] The Jacobian matrices at the equilibria are
\[
J(p^{(1)})=J(p^{(3)})=\begin{pmatrix}-1&-1\\-1&-1\end{pmatrix},
\qquad
J(p^{(2)})=\begin{pmatrix}-1&-1\\-1&-2\end{pmatrix}.
\]
The eigenvalues of the first matrix are $0$ and $-2$, and those of the
second are $(-3\pm\sqrt5)/2<0$.
\end{enumerate}
\end{proposition}

\begin{proof}
(a) Let $t:=y-2$.  Then $2-y=-t$ and $(y-1)(y-3)=t^2-1$, so
$\kappa(y)=-t(t^2-1)^2$.  Differentiating,
\[
\kappa'(y)=-(t^2-1)^2-4t^2(t^2-1)
=-(t^2-1)(5t^2-1)
=\frac45-5\Bigl(t^2-\frac35\Bigr)^2\leq\frac45 .
\]
Hence $k'(y)=-1+\kappa'(y)\leq-1/5$.

(b) Hypothesis (A1) holds because $f_{x_2}=g_{x_1}=-1$.  For
$x\in\mathbb R^2_{>0}$, $f(x)=0$ if and only if $x_2=4-x_1$, and then
$x_2>0$ if and only if $x_1<4$.  Thus (A2) holds with $a=4$, since $h$
maps $(0,4)$ into $(0,4)$ and $h'=-1$.  Next, $g(x)=0$ if and only if
$x_1=k(x_2)$.  By (a), $k$ is strictly decreasing on $\mathbb R$.  Since
$k(0)=4+\kappa(0)=22$, $k(3)=1$, and
$k(7/2)=1/2+\kappa(7/2)=1/2-75/32<0$, the function $k$ has a unique zero
$b$, which lies in $(3,7/2)$.  For $x_2>0$ we have $k(x_2)>0$ if and
only if $x_2<b$.  Hence the positive part of the second nullcline is
the graph of $k$ over $(0,b)$, and $k((0,b))=(0,22)\supset(0,4)$, which
is (A3).  Hypothesis (A4) follows from the sign information in~(c).

(c) Replacing $t$ by $-t$ shows that $\kappa(4-y)=-\kappa(y)$.  For
$0<x_1<4$ we therefore have
\[
k(h(x_1))=4-(4-x_1)+\kappa(4-x_1)=x_1-\kappa(x_1),
\qquad
k(\ell(x_1))=x_1 .
\]
Since $k$ is a polynomial, the fundamental theorem of calculus gives
\[
-\kappa(x_1)=k(h(x_1))-k(\ell(x_1))
=H(x_1)\int_0^1k'\bigl(\ell(x_1)+sH(x_1)\bigr)\,ds
=-\mu(x_1)H(x_1).
\]
By (a), $\mu\geq1/5$.  Moreover, $\ell$ is $C^\infty$ by the inverse
function theorem, since $k$ is a polynomial with $k'\ne0$; hence $H$ and
$\mu$ are $C^\infty$.  This proves
\eqref{eq:example-H}.  Since
$\kappa(x_1)=(2-x_1)(x_1-1)^2(x_1-3)^2$ and $\mu$ is smooth and
positive, the zeros, orders of vanishing, and signs of $H$ are those of
$\kappa$.

(d) By Lemma~\ref{lem:general-structure}(a) and~(c), the positive
equilibria are the points $(x_1,4-x_1)$ with $x_1\in\{1,2,3\}$.  By~(c),
$H$ changes sign at $2$ but not at $1$ or $3$.  Hence alternative~(ii)
holds, and $p^{(1)}$ and $p^{(3)}$ are degenerate by
Theorem~\ref{thm:general}.

(e) We have
\[
J(x)=\begin{pmatrix}-1&-1\\-1&-1+\kappa'(x_2)\end{pmatrix}.
\]
By the formula in~(a), at $y=1$ and $y=3$ we have $t^2=1$, so
$\kappa'=0$, while $\kappa'(2)=-1$.  This gives the three matrices.  Their
characteristic polynomials are $\lambda^2+2\lambda$ and
$\lambda^2+3\lambda+1$, and $\sqrt5<3$.
\end{proof}

The system \eqref{eq:example} has a symmetry that exchanges the two
outer equilibria.

\begin{lemma}\label{lem:example-symmetry}
Let $\rho(x):=(4-x_1,4-x_2)$.  Then $f\circ\rho=-f$ and $g\circ\rho=-g$
on $\mathbb R^2$.  Consequently, if $x(t)$ is a solution of
\eqref{eq:example}, then so is $\rho(x(t))$, on the same interval of
existence.  The map $\rho$ is an isometry, and it interchanges
$p^{(1)}$ and $p^{(3)}$.
\end{lemma}

\begin{proof}
We have $f(\rho(x))=h(4-x_1)-(4-x_2)=x_1+x_2-4=-f(x)$.  Using
$\kappa(4-y)=-\kappa(y)$,
\[
g(\rho(x))=k(4-x_2)-(4-x_1)=x_2-\kappa(x_2)-4+x_1=-g(x).
\]
If $x(t)$ is a solution, then
\[
\frac{d}{dt}\rho(x(t))=-\dot x(t)
=-\bigl(f(x(t)),g(x(t))\bigr)
=\bigl(f(\rho(x(t))),g(\rho(x(t)))\bigr),
\]
so $\rho(x(t))$ is a solution.  Since $\rho\circ\rho$ is the identity,
maximal intervals of existence correspond.  The last statement is
clear.
\end{proof}

Following~\cite[Definition~5.8.6]{mct}, an equilibrium $p$ is
\emph{unstable} if there is $\varepsilon>0$ such that, for every
$\delta>0$, some solution with $|x(0)-p|<\delta$ satisfies
$|x(T)-p|\geq\varepsilon$ at some time $T\geq0$ at which it is defined;
here $|\cdot|$ is the Euclidean norm.  An unstable equilibrium is, in
particular, not stable in the sense of Lyapunov.

\begin{proposition}\label{prop:example-stability}
For \eqref{eq:example}, the middle equilibrium $p^{(2)}=(2,2)$ is a
hyperbolic asymptotically stable node, whereas the outer equilibria
$p^{(1)}=(1,3)$ and $p^{(3)}=(3,1)$ are unstable.  In particular,
\eqref{eq:example} has exactly three positive equilibria but only one
stable equilibrium, so it is not bistable.
\end{proposition}

\begin{proof}
The assertion about $p^{(2)}$ follows from
Proposition~\ref{prop:example}(e) and the linearization principle.

\emph{Instability of $p^{(3)}=(3,1)$.}  We use the instability theorem
in~\cite{mct} (Theorem~20 there).  It requires a $C^1$ function $V$ on a
ball $\mathcal B$ centered at the equilibrium and a constant $c>0$ such
that $\dot V-cV\geq0$ on $\mathcal B$, where $\dot V$ denotes the
derivative of $V$ along solutions, and such that the equilibrium lies in
the closure of $\{V>0\}$.

Let $w(x):=x_1+x_2-4$, so that $f=-w$ and $g=-w+\kappa(x_2)$, and define
\[
\Psi(x):=(x_2-1)-\tfrac12w(x)-\tfrac14w(x)^2 .
\]
Then $\nabla\Psi=\bigl(-\tfrac12(1+w),\ \tfrac12(1-w)\bigr)$, and the
derivative of $\Psi$ along solutions is
\[
\dot\Psi
=\tfrac12(1+w)w+\tfrac12(1-w)\bigl(-w+\kappa(x_2)\bigr)
=w^2+\tfrac12(1-w)\kappa(x_2).
\]
The function $\Psi$ itself does not satisfy the hypothesis of the
theorem.  Indeed, at the points $(3-s,1+s)$ we have $\Psi=s$ and
$\dot\Psi=\tfrac12s^2(1-s)(2-s)^2$.  Thus, along this curve,
\[
\frac{\dot\Psi}{\Psi}\longrightarrow0
\qquad\text{as }s\downarrow0,
\]
so no inequality of the form $\dot\Psi-c\Psi\geq0$ with $c>0$ can hold
in a neighborhood of $(3,1)$.  We therefore compose $\Psi$ with a flat
function; this converts the quadratic lower bound
\eqref{eq:Psi-quadratic} for $\dot\Psi$, established below, into the
linear lower bound for $\dot V$ in terms of $V$ required by the
theorem.

Let $\mathcal B$ be the open disk of radius $1/2$ centered at $(3,1)$,
and write $\xi:=x_2-1$.  For $x\in\mathcal B$ we have $|\xi|<1/2$ and
$|w|=|(x_1-3)+\xi|\leq\sqrt2\,|x-(3,1)|<3/4$.  Hence $1-w>1/4$ and
\[
\kappa(x_2)=\xi^2(1-\xi)(2-\xi)^2\geq\tfrac98\,\xi^2,
\]
so that $\dot\Psi\geq w^2+\tfrac{9}{64}\xi^2$ on $\mathcal B$.  On the
other hand, $w^2/4\leq\tfrac{3}{16}|w|$ on $\mathcal B$, so
$|\Psi|\leq|\xi|+\tfrac{11}{16}|w|$ and therefore $\Psi^2\leq2\xi^2+w^2$.
Consequently,
\begin{equation}\label{eq:Psi-quadratic}
\dot\Psi\geq c\,\Psi^2\quad\text{on }\mathcal B,
\qquad c:=\frac{9}{128}.
\end{equation}

Now let $\psi(u):=e^{-1/u}$ for $u>0$ and $\psi(u):=0$ for $u\leq0$.  This
is a $C^\infty$ function on $\mathbb R$, with $\psi'(u)=\psi(u)/u^2$ for
$u>0$ and $\psi'(u)=0$ for $u\leq0$.  Set $V:=\psi\circ\Psi$, which is
$C^\infty$.  At points of $\mathcal B$ where $\Psi>0$,
\eqref{eq:Psi-quadratic} gives
\[
\dot V-cV=\psi(\Psi)\Bigl(\frac{\dot\Psi}{\Psi^2}-c\Bigr)\geq0,
\]
while at points where $\Psi\leq0$ both $V$ and $\dot V$ vanish.  Thus
$\dot V-cV\geq0$ on $\mathcal B$.  Finally, $V(3-s,1+s)=e^{-1/s}>0$ for
all sufficiently small $s>0$, and these points lie in $\mathcal B$ and
approach $(3,1)$, so $(3,1)$ lies in the closure of $\{V>0\}$.  By Theorem~20
of~\cite{mct}, the equilibrium $(3,1)$ is unstable.

\emph{Instability of $p^{(1)}=(1,3)$.}  Let $\varepsilon>0$ be as in the
definition of instability of $(3,1)$, and let $\delta>0$.  There are a
solution $x(t)$ with $|x(0)-(3,1)|<\delta$ and a time $T\geq0$ at which
it is defined such that $|x(T)-(3,1)|\geq\varepsilon$.  With $\rho$ as in
Lemma~\ref{lem:example-symmetry}, $\rho(x(t))$ is a solution, defined at
time $T$, with
\[
|\rho(x(0))-(1,3)|=|x(0)-(3,1)|<\delta,
\qquad
|\rho(x(T))-(1,3)|=|x(T)-(3,1)|\geq\varepsilon .
\]
Hence $(1,3)$ is unstable.
\end{proof}

\begin{remark}\label{rem:example-comments}
In \eqref{eq:example}, the outer equilibria, which in the setting of
Theorem~\ref{thm:bistability} would be asymptotically stable nodes, are
degenerate and unstable, and the unique stable equilibrium is the middle
one.  This failure is not caused by unbounded solutions.  Indeed, on the
four sides of the box $Q:=[0,9/2]\times[0,7/2]$, which contains the
three equilibria, the vector field points strictly into $Q$:
\begin{align*}
f&=4-x_2\geq\tfrac12 \text{ on } x_1=0,
&
f&=-\tfrac12-x_2<0 \text{ on } x_1=\tfrac92,\\
g&=22-x_1\geq\tfrac{35}{2} \text{ on } x_2=0,
&
g&=-\tfrac{59}{32}-x_1<0 \text{ on } x_2=\tfrac72 .
\end{align*}
Consequently $Q$ is forward invariant, that is, every solution with
$x(0)\in Q$ satisfies $x(t)\in Q$ for all $t\geq0$ in its interval of
existence (and, $Q$ being compact, this interval contains $[0,\infty)$).
Indeed, suppose that $x(0)\in Q$ but $x(t_1)\notin Q$ for some $t_1>0$,
and let
\[
\tau:=\sup\bigl\{t\in[0,t_1]:\ x(s)\in Q\ \text{for all } s\in[0,t]\bigr\}.
\]
Then $x(\tau)\in Q$ because $Q$ is closed, and hence $\tau<t_1$.  Each
coordinate $x_i(\tau)$ either lies in the interior of the corresponding
interval, $[0,9/2]$ or $[0,7/2]$, or is one of its endpoints.  In the
latter case, the strict inequalities above show that $\dot x_i(\tau)$
points into that interval.  By continuity, there is $\delta>0$ such that
$x(t)\in Q$ for all $t\in[\tau,\tau+\delta)$, which contradicts the
definition of $\tau$.  On the
other hand, (A5) does not hold on all of $\mathbb R^2_{>0}$.  Indeed,
$f(0,5)=-1<0$, so the solution through $(0,5)$ satisfies $x_1(T)<0$ for
some small $T>0$.  By continuous dependence on initial conditions, the
same holds for the solutions starting at $(\varepsilon,5)$ with
$\varepsilon>0$ small enough, and these solutions leave
$\mathbb R^2_{>0}$.  This plays no role, since instability is a local
property.
\end{remark}

\subsection{The two-gene system as a special case}

We now return to \eqref{eq:system1}--\eqref{eq:system2} and explain
Theorem~\ref{thm:bistability} in terms of Theorem~\ref{thm:general}.

\begin{proposition}\label{prop:special-case}
The system \eqref{eq:system1}--\eqref{eq:system2} satisfies (A1)--(A5),
with $h$ and $k$ as in Lemma~\ref{lem:nullclines}; the functions $\ell$
and $H$ of \eqref{eq:H-general} are the restrictions to $(0,a)$ of
\eqref{eq:ell-def} and \eqref{eq:H-def}.  If the system has exactly
three positive equilibria, then alternative~(ii) of
Theorem~\ref{thm:general} cannot occur, and all three equilibria are
nondegenerate.  Consequently, Theorem~\ref{thm:bistability} is a special
case of Theorem~\ref{thm:general}.
\end{proposition}

\begin{proof}
Hypothesis (A1) is \eqref{eq:offdiagonal}.  Hypotheses (A2) and (A3)
are contained in Lemma~\ref{lem:nullclines}; in particular, the image of
$k$ is all of $(0,\infty)$, which contains $(0,a)$.  Hypothesis (A4)
follows from \eqref{eq:H-end-signs}.  Hypothesis (A5) was established in
the proof of Theorem~\ref{thm:bistability}, through
\eqref{eq:upper-bounds}, \eqref{eq:lower-bound1}, and
\eqref{eq:lower-bound2}.

Suppose that there are exactly three positive equilibria.  By
Lemma~\ref{lem:polynomial}, the orders of vanishing $m_1,m_2,m_3$ of $H$
at the three zeros are the multiplicities of three distinct roots of a
nonzero polynomial of degree at most four, so $m_1+m_2+m_3\leq4$.
Moreover, $H$ changes sign at a zero if and only if its order of
vanishing is odd.  In alternative~(ii), two of the orders would be even,
hence at least two, and the third would be at least one, giving
$m_1+m_2+m_3\geq5$.  This is impossible, so alternative~(i) holds.  In
alternative~(i), all three orders are odd and their sum is at most four,
so $m_1=m_2=m_3=1$.  By Lemma~\ref{lem:jacobian}, all three equilibria
are nondegenerate.  The conclusions of Theorem~\ref{thm:bistability}
now follow from Theorem~\ref{thm:general}.
\end{proof}

\begin{remark}\label{rem:degree-bound}
For the present model, Lemma~\ref{lem:polynomial} is thus what
separates \eqref{eq:system1}--\eqref{eq:system2} from the example
\eqref{eq:example}; other additional hypotheses might exclude the
degenerate alternative as well.  At a zero where $H$ does not change sign, the order
of vanishing is even, hence at least two, so alternative~(ii) requires a
total order of vanishing of at least $2+1+2=5$.  For \eqref{eq:example},
formula \eqref{eq:example-H} shows that the orders are $(2,1,2)$, with
total exactly five.  For \eqref{eq:system1}--\eqref{eq:system2}, the
total is at most four.  With cooperative binding modeled by integer
Hill exponents larger than one, the equilibrium condition still reduces
to a polynomial equation, but of higher degree, and this degree
obstruction disappears; for noninteger Hill exponents, the equilibrium
condition is not polynomial.  We do not claim that the degenerate alternative is then
attained.  In any event, for such models Theorem~\ref{thm:general}
still applies whenever (A1)--(A4) can be verified.  Three positive
equilibria then yield the configuration of stable node, saddle, and
stable node as soon as nondegeneracy can be established, and global
bistability if (A5) holds as well.
\end{remark}

\subsubsection*{Acknowledgment}

The author worked closely with several AI assistants to brainstorm
ideas, search the literature, develop proofs, check arguments, and write
and edit the paper.  The author verified all mathematical statements and
proofs and takes full responsibility for the content.

\end{document}